\documentclass[11pt]{article}
\usepackage[T1]{fontenc}
\usepackage{amsfonts}
\usepackage{amsmath}
\usepackage{amssymb}
\usepackage{amsthm}
\usepackage{bbm}
\usepackage{bm}
\usepackage{mathrsfs}
\usepackage{color}
\usepackage{pdfsync}
\usepackage{enumitem}

\usepackage{tikz}

\DeclareMathOperator*{\argmin}{argmin}
\DeclareMathOperator*{\argmax}{argmax}

\DeclareMathOperator*{\spt}{spt}

\DeclareMathOperator*{\relint}{ri}
\DeclareMathOperator*{\rbd}{rbd}
\DeclareMathOperator*{\aff}{aff}
\DeclareMathOperator*{\ext}{ext}

\DeclareMathOperator*{\coh}{conv}

\DeclareMathOperator{\proj}{proj}

\DeclareMathOperator{\lin}{lin}
\DeclareMathOperator{\maxtr}{\max tr}
\newcommand{\Dom}{D}
\newcommand{\RR}{\mathbb{R}}
\newcommand{\R}{\RR}

\newcommand{\NN}{\mathbb{N}}

\newcommand{\eps}{\varepsilon}

\newcommand{\cP}{\mathcal{P}}

\usepackage[margin=31 mm]{geometry}
\newcommand{\mykill}[1]{}
\usepackage[pdfborder={0 0 0}]{hyperref}
\hypersetup{
  urlcolor = black,
  pdfauthor = {Alberto Gonzalez-Sanz, Marcel Nutz, Andres Riveros Valdevenito}
  pdfkeywords = {Linear Program, Quadratic Regularization, Optimal Transport, Sparsity},
  pdftitle = {Monotonicity in Quadratically Regularized Linear Programs},
  pdfsubject = {Monotonicity in Quadratically Regularized Linear Programs},
  pdfpagemode = UseNone
}
\usepackage[capitalize]{cleveref}

\theoremstyle{plain}
\newtheorem{theorem}{Theorem}[section]

\newtheorem{lemma}[theorem]{Lemma}
\newtheorem{corollary}[theorem]{Corollary}
\theoremstyle{definition}
\newtheorem{definition}[theorem]{Definition}
\newtheorem{remark}[theorem]{Remark}
\newtheorem{example}[theorem]{Example}
\newtheorem{question}[theorem]{Question}
\theoremstyle{remark}

\newlist{myenum}{enumerate}{3}
\setlist[myenum,1]{label={\rm (H\arabic*)},
                   ref  ={\rm (H\arabic*)}}
\crefname{myenumi}{property}{properties}

\renewenvironment{thebibliography}[1]{%
\begin{oldthebibliography}{#1}%
\setlength{\baselineskip}{.9em}
\linespread{1}
\small
\setlength{\parskip}{.40ex}%
\setlength{\itemsep}{.30em}%
}%
{%
\end{oldthebibliography}%
}

\begin{document}

\title{\vspace{-2.2em}
Monotonicity in Quadratically Regularized Linear Programs\footnote{The authors thank three anonymous referees for their careful reading and constructive comments.}}
\date{\today}
\author{
  Alberto Gonz{\'a}lez-Sanz%
  \thanks{Department of Statistics, Columbia University, ag4855@columbia.edu} \and 
  Marcel Nutz%
  \thanks{Departments of Mathematics and Statistics, Columbia University, mnutz@columbia.edu. Research supported by NSF Grants DMS-1812661, DMS-2106056, DMS-2407074.}
  \and
  Andr{\'e}s Riveros Valdevenito%
  \thanks{Department of Statistics, Columbia University, ar4151@columbia.edu.}
  }
  
\maketitle \vspace{-1.5em}

\begin{abstract}
In optimal transport, quadratic regularization is a sparse alternative to entropic regularization: the solution measure tends to have small support. Computational experience suggests that the support decreases monotonically to the unregularized counterpart as the regularization parameter is relaxed. We find it useful to investigate this monotonicity more abstractly for linear programs over polytopes, regularized with the squared norm. Here, monotonicity can be stated as an invariance property of the curve mapping the regularization parameter to the solution: once the curve enters a face of the polytope, does it remain in that face forever? We show that this invariance is equivalent to a geometric property of the polytope, namely that each face contains the minimum norm point of its affine hull. Returning to the optimal transport problem and its associated Birkhoff polytope, we verify this property for low dimensions, but show that it fails for marginals with five or more point masses. As a consequence, the conjectured monotonicity of the support \emph{fails} in general, even if experiments suggest that monotonicity holds for many cost matrices. Separately, we apply our geometric point of view to a problem of Erd{\H o}s, namely to characterize the doubly stochastic matrices whose maximal trace equals their squared norm.
\end{abstract}

\vspace{1em}

{\small


}
\vspace{0em}

\section{Introduction}\label{intro}

Let $\cP\subset\R^d$ be a polytope and $c\in\R^{d}$. The regularized linear program
\begin{equation}\label{eq:EulerIntro} 
     x^{\delta}=\argmin_{x \in \cP: \, \| x \|\leq \delta} \langle c, x \rangle
\end{equation}
has a unique solution as long as $\delta\geq0$ belongs to a certain interval~$[\delta_{\min},\delta_{\max}]$, cf.\ \cref{Lemma:ContinuityIncreasing}, hence we can consider the curve $\delta\mapsto x^{\delta}\in\cP$ which travels across various faces of~$\cP$ as $\delta$ increases (i.e., as the regularizing constraint relaxes). We are interested in the following invariance property: once $\delta\mapsto x^{\delta}\in\cP$ enters a given face, does it ever leave that face? (Cf.\ Figure~\ref{fig:invariancePropertyExample}.) In the optimal transport applications that motivate our study (see below), this property corresponds to the \emph{monotonicity of the optimal support $\spt x^{\delta}$ wrt.\ the regularization strength.} (The support is defined as the set of locations with nonzero mass, $\spt x := \{i: \, x_{i}>0\}$ for $x=(x_{1},\dots,x_{d})\in\R^{d}$.) Our abstract result, \cref{th:main}, geometrically characterizes all polytopes such that the invariance property holds (for any cost~$c$). 
We show that this property holds for the set of probability measures (unit simplex), but fails for the set of transport plans (Birkhoff polytope) when the dimension $d\geq25$. As a consequence---which may be surprising given numerical experience---the optimal support in quadratically regularized optimal transport problems is \emph{not} always monotone. 

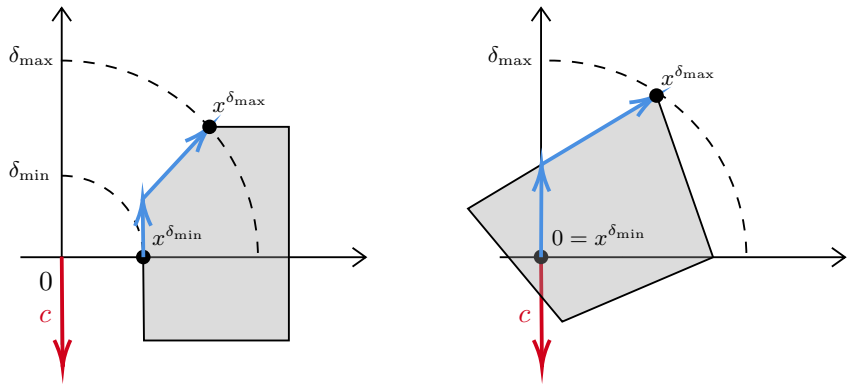
\begin{figure}
    \centering
    \resizebox{.75\linewidth}{!}{
    \tikzset{every picture/.style={line width=0.75pt}} %

    \begin{tikzpicture}[x=0.75pt,y=0.75pt,yscale=-1,xscale=1]
    \draw  (92.61,165.8) -- (275.77,165.8)(114.52,33.48) -- (114.52,212.29) (268.77,160.8) -- (275.77,165.8) -- (268.77,170.8) (109.52,40.48) -- (114.52,33.48) -- (119.52,40.48)  ;
    \draw  [fill={rgb, 255:red, 155; green, 155; blue, 155 }  ,fill opacity=0.35 ] (192.92,96.59) -- (235,96.59) -- (235,210) -- (158.14,210) -- (157.43,134.92) -- cycle ;
    \draw [color={rgb, 255:red, 208; green, 2; blue, 27 }  ,draw opacity=1 ][line width=1.5]    (114.52,165.48) -- (114.98,221) ;
    \draw [shift={(115,224)}, rotate = 269.53] [color={rgb, 255:red, 208; green, 2; blue, 27 }  ,draw opacity=1 ][line width=1.5]    (14.21,-4.28) .. controls (9.04,-1.82) and (4.3,-0.39) .. (0,0) .. controls (4.3,0.39) and (9.04,1.82) .. (14.21,4.28)   ;
    \draw  [draw opacity=0][dash pattern={on 4.5pt off 4.5pt}] (114.33,122.54) .. controls (114.33,122.54) and (114.33,122.54) .. (114.33,122.54) .. controls (138.22,122.43) and (157.68,141.72) .. (157.78,165.61) -- (114.52,165.8) -- cycle ; \draw  [dash pattern={on 4.5pt off 4.5pt}] (114.33,122.54) .. controls (114.33,122.54) and (114.33,122.54) .. (114.33,122.54) .. controls (138.22,122.43) and (157.68,141.72) .. (157.78,165.61) ;  
    \draw    (157.82,165.77) -- (157.78,165.8) ;
    \draw    (157.78,165.61) -- (157.78,165.8) ;
    \draw [shift={(157.78,165.71)}, rotate = 89.88] [color={rgb, 255:red, 0; green, 0; blue, 0 }  ][fill={rgb, 255:red, 0; green, 0; blue, 0 }  ][line width=0.75]      (0, 0) circle [x radius= 3.35, y radius= 3.35]   ;
    \draw [color={rgb, 255:red, 74; green, 144; blue, 226 }  ,draw opacity=1 ][line width=1.5]    (157.82,165.77) -- (157.47,137.92) ;
    \draw [shift={(157.43,134.92)}, rotate = 89.27] [color={rgb, 255:red, 74; green, 144; blue, 226 }  ,draw opacity=1 ][line width=1.5]    (14.21,-4.28) .. controls (9.04,-1.82) and (4.3,-0.39) .. (0,0) .. controls (4.3,0.39) and (9.04,1.82) .. (14.21,4.28)   ;
    \draw [color={rgb, 255:red, 74; green, 144; blue, 226 }  ,draw opacity=1 ][line width=1.5]    (157.43,134.92) -- (190.88,98.79) ;
    \draw [shift={(192.92,96.59)}, rotate = 132.8] [color={rgb, 255:red, 74; green, 144; blue, 226 }  ,draw opacity=1 ][line width=1.5]    (14.21,-4.28) .. controls (9.04,-1.82) and (4.3,-0.39) .. (0,0) .. controls (4.3,0.39) and (9.04,1.82) .. (14.21,4.28)   ;
    \draw  [draw opacity=0][dash pattern={on 4.5pt off 4.5pt}] (114.06,61.42) .. controls (171.54,61.17) and (218.33,107.55) .. (218.58,165.03) -- (114.52,165.48) -- cycle ; \draw  [dash pattern={on 4.5pt off 4.5pt}] (114.06,61.42) .. controls (171.54,61.17) and (218.33,107.55) .. (218.58,165.03) ;  
    \draw    (192.92,96.59) ;
    \draw [shift={(192.92,96.59)}, rotate = 0] [color={rgb, 255:red, 0; green, 0; blue, 0 }  ][fill={rgb, 255:red, 0; green, 0; blue, 0 }  ][line width=0.75]      (0, 0) circle [x radius= 3.35, y radius= 3.35]   ;

    \draw  (346.83,165.8) -- (530,165.8)(368.75,33.48) -- (368.75,212.29) (523,160.8) -- (530,165.8) -- (523,170.8) (363.75,40.48) -- (368.75,33.48) -- (373.75,40.48)  ;
    \draw [color={rgb, 255:red, 208; green, 2; blue, 27 }  ,draw opacity=1 ][line width=1.5]    (368.75,165.48) -- (369.2,221) ;
    \draw [shift={(369.23,224)}, rotate = 269.53] [color={rgb, 255:red, 208; green, 2; blue, 27 }  ,draw opacity=1 ][line width=1.5]    (14.21,-4.28) .. controls (9.04,-1.82) and (4.3,-0.39) .. (0,0) .. controls (4.3,0.39) and (9.04,1.82) .. (14.21,4.28)   ;
    \draw    (412.05,165.77) -- (412.01,165.8) ;
    \draw    (368.75,165.48) -- (368.75,165.8) ;
    \draw [shift={(368.75,165.64)}, rotate = 90] [color={rgb, 255:red, 0; green, 0; blue, 0 }  ][fill={rgb, 255:red, 0; green, 0; blue, 0 }  ][line width=0.75]      (0, 0) circle [x radius= 3.35, y radius= 3.35]   ;
    \draw  [draw opacity=0][dash pattern={on 4.5pt off 4.5pt}] (373.16,61.42) .. controls (430.63,61.17) and (477.42,107.55) .. (477.67,165.03) -- (373.61,165.48) -- cycle ; \draw  [dash pattern={on 4.5pt off 4.5pt}] (373.16,61.42) .. controls (430.63,61.17) and (477.42,107.55) .. (477.67,165.03) ;  
    \draw  [fill={rgb, 255:red, 155; green, 155; blue, 155 }  ,fill opacity=0.35 ] (430,80) -- (460,165.8) -- (380,200) -- (330,140) -- cycle ;
    \draw [color={rgb, 255:red, 74; green, 144; blue, 226 }  ,draw opacity=1 ][line width=1.5]    (368.75,165.8) -- (368.98,119.6) ;
    \draw [shift={(369,116.6)}, rotate = 90.29] [color={rgb, 255:red, 74; green, 144; blue, 226 }  ,draw opacity=1 ][line width=1.5]    (14.21,-4.28) .. controls (9.04,-1.82) and (4.3,-0.39) .. (0,0) .. controls (4.3,0.39) and (9.04,1.82) .. (14.21,4.28)   ;
    \draw [color={rgb, 255:red, 74; green, 144; blue, 226 }  ,draw opacity=1 ][line width=1.5]    (369,116.6) -- (427.43,81.54) ;
    \draw [shift={(430,80)}, rotate = 149.04] [color={rgb, 255:red, 74; green, 144; blue, 226 }  ,draw opacity=1 ][line width=1.5]    (14.21,-4.28) .. controls (9.04,-1.82) and (4.3,-0.39) .. (0,0) .. controls (4.3,0.39) and (9.04,1.82) .. (14.21,4.28)   ;
    \draw    (430,80) ;
    \draw [shift={(430,80)}, rotate = 0] [color={rgb, 255:red, 0; green, 0; blue, 0 }  ][fill={rgb, 255:red, 0; green, 0; blue, 0 }  ][line width=0.75]      (0, 0) circle [x radius= 3.35, y radius= 3.35]   ;

    \draw (101,172.4) node [anchor=north west][inner sep=0.75pt]    {$0$};
    \draw (101,192.4) node [anchor=north west][inner sep=0.75pt]    {$\textcolor[rgb]{0.82,0.01,0.11}{c}$};
    \draw (85,113.4) node [anchor=north west][inner sep=0.75pt]  [font=\footnotesize]  {$\delta_{\min}$};
    \draw (161,147.4) node [anchor=north west][inner sep=0.75pt]  [font=\footnotesize]  {$x^{\delta_{\min}}$};
    \draw (85,52.4) node [anchor=north west][inner sep=0.75pt]  [font=\footnotesize]  {$\delta_{\max}$};
    \draw (193,76.4) node [anchor=north west][inner sep=0.75pt]  [font=\footnotesize]  {$x^{\delta_{\max}}$};
    \draw (355.23,192.4) node [anchor=north west][inner sep=0.75pt]    {$\textcolor[rgb]{0.82,0.01,0.11}{c}$};
    \draw (373,147) node [anchor=north west][inner sep=0.75pt]  [font=\footnotesize]  {$0=x^{\delta_{\min}}$};
    \draw (339,52.4) node [anchor=north west][inner sep=0.75pt]  [font=\footnotesize]  {$\delta_{\max}$};
    \draw (431,62.4) node [anchor=north west][inner sep=0.75pt]  [font=\footnotesize]  {$x^{\delta_{\max}}$};

    \end{tikzpicture}
    }
    \caption{A non-monotone (left) and a monotone (right) polytope with cost $c$ (red). The curve $[\delta_{\min},\delta_{\max}]\ni \delta\mapsto x^{\delta}$ follows the blue arrows and has the invariance property only in the right example.}
    \label{fig:invariancePropertyExample}
\end{figure}

\subsection{Motivation}

This study is motivated by optimal transport  and related minimization problems over probability measures. In its simplest form, the transport problem between probability measures~$\mu$ and~$\nu$ is
\begin{equation}\tag{OT}\label{eq:OT}
\inf_{\gamma \in \Gamma(\mu, \nu)} \int \hat c(x,y) \, \gamma(dx,dy),
\end{equation}
where $\hat c(x,y)$ is a given ``cost'' function and $\Gamma(\mu, \nu)$ denotes the set of couplings; i.e., joint probability measures $\gamma$ with marginals $(\mu,\nu)$. See~\cite{Villani.09} for a detailed discussion. In many applications such as machine learning or statistics (see \cite{KolouriParlEtAl.17survey,PanaretosZemel.19} for surveys), the marginals encode observed data points ${X_1, \dots, X_N}$ and ${Y_1, \dots, Y_N}$ which are represented by their empirical measures $\mu=\frac{1}{N} \sum_{i} \delta_{X_i}$ and $\nu=\frac{1}{N} \sum_{i} \delta_{Y_i}$.   Denoting by $c_{ij}=\hat c(X_i,Y_j)$ the resulting $N\times N$ cost matrix, \eqref{eq:OT} then becomes a linear program 
\begin{align}\tag{LP}\label{eq:LP}
\inf_{x \in \cP} \langle c, x \rangle
\end{align} 
where the polytope $\cP$ is (up to a constant factor) the Birkhoff polytope of doubly stochastic matrices of size $N\times N$.   
For different choices of polytope, \eqref{eq:LP} includes other problems of recent interest, such as multi-marginal optimal transport and Wasserstein barycenters~\cite{AguehCarlier.11} or adapted Wasserstein distances~\cite{BackhoffBartlBeiglbockEder.20}. The optimal transport problem is computationally costly when $N$ is large. The impactful paper \cite{Cuturi.13}  proposed to regularize~\eqref{eq:OT} by penalizing with Kullback--Leibler divergence (entropy). Then, solutions can be computed using the Sinkhorn--Knopp algorithm, which has lead to an explosion of high-dimensional applications (see~\cite{PeyreCuturi.19}). More generally, \cite{DesseinPapadakisRouas.18} introduced regularized optimal transport with regularization by a divergence. Different divergences give rise to different properties of the solution. Entropic regularization always leads to couplings whose support contains all data pairs $(X_i,Y_j)$, even though~\eqref{eq:OT}  typically has a sparse solution. In some applications that is undesirable; for instance, it may correspond to blurrier images in an image processing task~\cite{blondel18quadratic}. The second-most prominent regularization is $\chi^{2}$-divergence or equivalently the squared norm, as proposed in \cite{blondel18quadratic},\footnote{See also \cite{EssidSolomon.18} for  a similar formulation of minimum-cost flow problems, and the predecessors referenced therein. Our model with a general polytope includes such problems, and many others.} which gives rise to sparse solutions. In the Eulerian formulation of \cite{DesseinPapadakisRouas.18}, this is exactly our problem~\eqref{eq:EulerIntro} with $\cP$ being the (scaled) Birkhoff polytope. Alternately, the problem can be stated in Lagrangian form as in \cite{blondel18quadratic}, making the squared-norm penalty explicit:
\begin{equation}\tag{QOT}\label{eq:QOTintro}
    \inf_{\gamma\in \Gamma(\mu, \nu)} \int \hat c (x,y)\, \gamma(dx,dy)+\frac{1}{2\eta}  \left\| \frac{d\gamma}{  d(\mu \otimes \nu) }\right\|_{L^2( \mu \otimes \nu)}^2
\end{equation} 
where $d\gamma/  d(\mu \otimes \nu) $ denotes the density of $\gamma$ wrt.\ the  product measure $\mu \otimes \nu$ and the regularization strength is now parameterized by $\eta\in[0,\infty]$. In the general setting, this corresponds to
\begin{align}\label{eq:LagrangeIntro}
    \inf_{x \in \cP} \langle c, x \rangle + \frac{1}{2\eta} \| x \|^2
\end{align} 
which has a unique solution $x_{\eta}$ for any $\eta\in(0,\infty)$. The curves $(x^{\eta})$ and $(x_{\eta})$ of solutions to~\eqref{eq:EulerIntro} and~\eqref{eq:LagrangeIntro}, respectively, coincide up to a simple reparametrization.

In optimal transport, the regularized problem is often solved to approximate the linear problem~\eqref{eq:OT}. The latter has a generically unique (see \cite{CuestaAlbertosTueroDiaz.93}) solution $x$, which is recovered as $\delta\to\delta_{\max}$ (or $\eta\to\infty$): $x=x^{\delta_{\max}}=x_{\infty}$.\footnote{By ``generic'' we mean that the set of $c$ where uniqueness fails is a Lebesgue nullset. For problems with non-unique solution, $x^{\delta_{\max}}$ recovers the minimum-norm solution.} In particular, the support of $x^{\delta}$ converges to the support of $x$, which is generically sparse ($N$ out of $N^{2}$ possible points, again by~\cite{CuestaAlbertosTueroDiaz.93}). In numerous experiments, it has been observed not only that $\spt x^{\delta}$ is sparse when~$\delta$ is large, but also that $\spt x^{\delta}$ monotonically decreases to $\spt x$ (e.g., \cite{blondel18quadratic}). If this monotonicity holds, then in particular $\spt x\subset \spt x^{\delta}$, meaning that $\spt x^{\delta}$ can be used as a multivariate confidence band for the (unknown) solution~$x$. 

\subsection{Summary of Results}

When $\cP$ is the unit simplex or the Birkhoff polytope, its elements can be interpreted as probability measures and thus carry a natural notion of support, $\spt x := \{i: \, x_{i}>0\}$ for $x=(x_{1},\dots,x_{d})\in\R^{d}$. We can then ask whether $\spt x^{\delta}$ is monotone wrt.\ the strength~$\delta$ of regularization. When $\cP$ is a general polytope, the set of vertices of the minimal face containing~$x$ yields a similar notion, replacing the concept of support. The question of monotonicity then becomes to the aforementioned invariance property: if $x^{\delta}\in F$ for some face~$F$, does it follow that $x^{\delta'}\in F$ for all $\delta'\geq\delta$? (This indeed specializes to the monotonicity of $\spt x^{\delta}$ when $\cP$ is the simplex or the Birkhoff polytope; cf.\ \cref{le:BirkhoffSptMonEquiv}.) The answer may of course depend on the cost~$c$; we call~$\cP$ monotone if the invariance holds for any $c\in\RR^{d}$.

We show that monotonicity can be characterized by the geometry of~$\cP$. Namely, monotonicity of~$\cP$ is equivalent to two properties: for any proper face~$F$, the minimum-norm point of the affine hull of~$F$ must lie in~$F$, and moreover, the minimum-norm point of~$\cP$ must lie in the relative interior $\relint\cP$. See \cref{th:main}. Once the right point of view is taken, the proof is quite elementary.  An example satisfying both properties, and hence monotonicity, is the unit simplex $\Delta\subset\R^{d}$, for any $d\geq1$. For that choice of polytope, $x^{\delta}$ is a sparse soft-min of $c=(c_{1},\dots,c_{d})$, converging to $\frac{1}{\#\argmin c} \; 1_{\{\argmin c\}}$ in the unregularized limit. On the other hand, we show that the Birkhoff polytope violates the first condition whenever the marginals have $N\geq 5$ data points (meaning that $d\geq 25$). As a result, the optimal support in quadratically regularized optimal transport problems is \emph{not} always monotone, even if it appears so in numerous experiments. In fact, we exhibit a particularly egregious failure of monotonicity where the support of the limiting (unregularized) optimal coupling~$x$ is not contained in $\spt x^{\delta}$ for some reasonably large~$\delta$. Our counterexample is constructed by hand, based on theoretical considerations. Our numerical experiments using random cost matrices have failed to locate counterexamples, suggesting that there are, in some sense, ``few'' faces violating our condition on the minimum-norm point. (See also the proof of \cref{le:BirkhoffMonotone}.) 

This is an interesting problem for further study in combinatorics, where the Birkhoff polytope has remained an active topic even after decades of research (e.g., \cite{Paffenholz.15} and the references therein). In that spirit, we also apply our point of view to a problem of Erd{\H o}s, namely to characterize the doubly stochastic $N\times N$ matrices $A$ satisfying $\|A\|^{2}= \maxtr A := \max_\sigma \sum_{i=1}^N a_{i,\sigma(i)}$
where $\sigma$ ranges over all permutations of $\{1,\dots,N\}$. We show in \cref{th:erd} that this is precisely the set of minimum-norm points of certain faces $F$ of the Birkhoff polytope. We use this geometric characterization to prove a conjecture of~\cite{BouthatMashreghiMorneauGuerin.24}, namely that such matrices must have rational entries.

The remainder of this note is organized as follows. \Cref{se:prelims} introduces the regularized linear program and its solution curve in detail. \Cref{se:main} contains the main abstract result, whereas \cref{se:applications} reports the applications to optimal transport and soft-min. We conclude in \cref{se:erd} with the application to Erd{\H o}s' problem about doubly stochastic matrices.

\section{Preliminaries}\label{se:prelims}

This section collects notation and well-known or straightforward results for ease of reference. Let $\langle  \cdot,\cdot \rangle$ be an inner product on $\RR^{d}$ and $\|\cdot\|$ the induced norm. Let $\emptyset\neq\cP \subseteq \RR^{d}$ be a polytope; i.e., the convex hull of finitely many points. The minimal set of such points, called vertices or extreme points, is denoted  $\ext \cP$. A face of $\cP$ is a subset $F\subset \cP$ such that any open line segment $L=(x,y)\subset\cP$ with $L\cap F\neq\emptyset$ satisfies $\overline L \subset F$, where $\overline L=[x,y]$ denotes closure. Alternately, a nonempty face~$F$ of $\cP$ is the intersection $F=\cP\cap H$ of~$\cP$ with a tangent hyperplane~$H$. Here a hyperplane $H=\{x\in \RR^d: \ \langle a, x \rangle=b\}$ is called tangent if $H\cap \cP\neq \emptyset$ and $ \cP\subset  \{x\in \RR^d: \ \langle a, x \rangle\leq b\}$.
We denote by $\relint K$ and $\rbd K = K\setminus \relint K$ the relative interior and relative boundary of a set~$K$, and by $\aff K$ its affine span. See \cite{Brondsted.83} for background on polytopes and related notions. 

Next, we introduce the regularized linear program and the interval of parameters $\delta$ where the set of feasible points is nonempty and the constraint is binding.

\begin{lemma}[Eulerian Formulation]\label{Lemma:ContinuityIncreasing}
Denote $\cP(\delta):=\{x\in\cP:\, \|x\|\leq\delta\}$ for $\delta\geq0$ and
$$
  \delta_{\min}:=\min\{\delta\geq0:\, \cP(\delta)\neq\emptyset\}, \qquad 
  \delta_{\max}:=\min\left\{\delta\geq0:\, \min_{x\in\cP(\delta)} \langle c,x\rangle = \min_{\ x\in\cP} \langle c,x\rangle\right\}.
$$
Then
$$
 \Dom := \left\{\delta>\delta_{\min}: \ \argmin_{x\in\cP(\delta)} \langle c,x\rangle\cap  \argmin_{\ x\in\cP} \langle c,x\rangle=\emptyset \right\} = (\delta_{\min},\delta_{\max}).
$$
For each $\delta\in \overline{\Dom }=[\delta_{\min},\delta_{\max}]$, the problem 
\begin{align}\label{eq:EulerianSol}
  \inf_{x\in\cP(\delta)} \langle c,x\rangle \qquad \text{has a unique minimizer $x^{\delta}$.}
\end{align}
Moreover, $[\delta_{\min},\delta_{\max}]\ni \delta\mapsto x^{\delta}$ is continuous and $\|x^{\delta}\|=\delta$.
In particular, $x^{\delta_{\min}}$ is the singleton $\cP(\delta_{\min})$, and $x^{\delta_{\max}}$ is the minimum-norm solution of $\min_{x\in\cP} \langle c,x\rangle$.
\end{lemma}

\begin{proof}
Compactness of $\cP$ implies that the minima defining $\delta_{\min}$ and $\delta_{\max}$ are indeed attained. Moreover, the identity $\Dom =(\delta_{\min},\delta_{\max})$ follows directly from the definition of $\delta_{\max}$.

Fix $\delta\in \Dom $. Let $x^{\delta}\in \argmin_{x\in\cP(\delta)} \langle c,x\rangle $ and $x^{*}\in \argmin_{x\in\cP} \langle c,x\rangle$. If $\|x^{\delta}\|<\delta$, then $x:=\lambda x^{\delta} + (1-\lambda) x^{*}\in \cP(\delta)$ for sufficiently large $\lambda\in(0,1)$. By optimality it follows that $x^{\delta}\in \argmin_{x\in\cP} \langle c,x\rangle $, meaning that $\delta\notin\Dom $. Thus, for $\delta\in\Dom $, we have $\|x^{\delta}\|=\delta$. In particular, the set $\argmin_{x\in\cP(\delta)} \langle c,x\rangle$ is contained in the sphere $\{x:\,\|x\|=\delta\}$. As the set is also convex, it must be a singleton by the strict convexity of $\|\cdot\|$. Continuity of $\delta\mapsto x^{\delta}$ is straightforward.

It remains to deal with the boundary cases. For $\delta=\delta_{\min}$, we note that $\|x^{\delta}\|<\delta$ is trivially ruled out, hence we can conclude as above. Clearly $\{x^{\delta_{\min}}\}= \cP(\delta_{\min})$. Define $x^{*}$ as the minimum-norm solution of $\min_{x\in\cP} \langle c,x\rangle$, which is unique by strict convexity. Clearly $\|x^{*}\|=\delta_{\max}$. Thus when $\delta=\delta_{\max}$, we must have $x^{\delta}=x^{*}$ for any $x^{\delta}\in \argmin_{x\in\cP(\delta)} \langle c,x\rangle$.
Continuity and $\|x^{\delta}\|=\delta$ are again straightforward.
\end{proof}

The next lemma recalls the standard projection theorem (e.g., \cite[Theorem~5.2]{Brezis.11}).

\begin{lemma}[Projection]
    \label{Lemma:projset}
     Let $\emptyset\neq K \subseteq \RR^d$ be closed and convex. Given $x \in \RR^d$, there exists a unique $x_{K} \in K$, called the projection of $x$ onto $K$ and denoted $x_{K}=\proj_{K}(x)$, such that
    \begin{align*}
        \|x - x_{K}\| = \inf_{x'\in K} \|x - x'\|.
    \end{align*}
    Moreover, $x_{K}$ is characterized within~$K$ by 
    \begin{align}\label{projineq}
        \langle x - x_{K}, x' - x_{K} \rangle \leq 0 \qquad \text{for all } x'\in K.
    \end{align}
   If $x_{K}\in \relint K$, then $x_{K}=\proj_{\aff K}(x)$ and~\eqref{projineq} can be sharpened to
    \begin{align}\label{projeq}
        \langle x - x_{K}, x' - x_{K} \rangle = 0 \qquad \text{for all } x'\in K.
    \end{align}
    In particular, \eqref{projeq} holds when $K$ is an affine subspace. In that case, $x\mapsto\proj_{K}(x)$ is affine.
\end{lemma}

Below, we often use $\proj_{K}(0)$ as a convenient notation for $\argmin_{x\in K} \|x\|$, the minimum-norm point of~$K$.
In the Lagrangian formulation of our regularized linear program, the solution can be expressed as a projection onto~$\cP$ as follows.

\begin{lemma}[Lagrangian Formulation]\label{le:Lagrange}
  Define 
  \begin{align}\label{eq:xetaProj}
     x_{\eta} := \proj_{\cP}(-\eta c), \quad \eta\in[0,\infty), \qquad x_{\infty}:=\lim_{\eta\to\infty} x_{\eta}. 
  \end{align} 
  Then 
  \begin{align}
     x_{\eta} &= \argmin_{x\in\cP} \, \langle c,x \rangle + \frac{1}{2\eta}\|x\|^{2} , \quad \eta \in (0,\infty), \label{eq:xeta2} \\
	x_{0} &= \argmin_{x\in\cP} \|x\|,  \qquad 
	x_{\infty} = \argmin_{x'\in\argmin_{x\in\cP} \langle c,x \rangle} \|x'\|. \label{eq:xeta3}
  \end{align}
  The limit $x_{\eta}\to x_{\infty}$ is stationary; i.e., there exists $\bar\eta\in\RR_{+}$ such that $x_{\eta}=x_{\infty}$ for all $\eta\geq \bar\eta$.
\end{lemma}

\begin{proof}
  Let $\eta\in(0,\infty)$. Then~\eqref{eq:xeta2} follows from
  $$
    \frac{1}{2\eta}\|-\eta c - x\|^{2} = \frac{\eta}{2} \|c\|^{2} + \langle c,x \rangle + \frac{1}{2\eta}\|x\|^{2}.
  $$
  The first claim in~\eqref{eq:xeta3} is trivial. For the second claim and the stationary convergence $x_{\eta}\to x_{\infty}$ (which will not be used directly), see \cite[Theorem~2.1]{Mangasarian.84}, or \cite{GonzalezSanzNutz.24} for the exact threshold~$\bar{\eta}$.
\end{proof} 

The algorithm of~\cite{HagerZhang.16} solves the problem of projecting a point onto a polyhedron, hence can be used to find~$x_{\eta}=\proj_{\cP}(-\eta c)$ numerically. The solutions $x^{\delta}$ of the Eulerian formulation~\eqref{eq:EulerianSol} and $x_{\eta}$ of the Lagrangian formulation~\eqref{eq:xetaProj} are related as follows.

\begin{lemma}[Euler $\leftrightarrow$ Lagrange]\label{le:EulerLagrange}
  Given $\eta\in[0,\infty]$, there exists a unique $\delta=\delta(\eta)\in\overline{\Dom }$ such that $x_{\eta}=x^{\delta}$, namely $\delta=\|x_{\eta}\|$. The function $\eta\mapsto\delta(\eta)$ is nondecreasing.

  Conversely, let $\delta\in\overline{\Dom }$. Then there exists $\eta\in[0,\infty)$ (possibly non-unique) such that $x_{\eta}=x^{\delta}$. Define $\eta(\delta)$ as the minimal $\eta$ with $x_{\eta}=x^{\delta}$.\footnote{This is merely for concreteness; any choice of selector will do.} Then $\delta\mapsto \eta(\delta)$ is strictly increasing on $\overline{\Dom }$.
\end{lemma}

\begin{proof}
We first prove the second part. Let $\delta\in\overline{\Dom}$ and recall that $\|x^{\delta}\|=\delta$. Note also that $\eta\mapsto x_{\eta}$ is continuous and that $\eta\mapsto \|x_{\eta}\|$ is nondecreasing, with range $[\|x_{0}\|, \|x_{\infty}\|] = \overline{\Dom }$. Hence there exists $\eta=\eta(\delta)$ such that $\|x_{\eta}\|=\delta$. We see from~\eqref{eq:xeta2} that $x_{\eta}$ minimizes $\langle c,x \rangle$ among all $x\in\cP(\delta)$, which is to say that $x_{\eta}=x^{\delta}$. 

To see the first part, fix $\eta\in[0,\infty]$ and define $\delta(\eta):=\|x_{\eta}\|$. The above shows that $x_{\eta}=x^{\delta(\eta)}$. Moreover, $\delta(\eta)$ must be unique with that property, due to $\|x^{\delta}\|=\delta$.
\end{proof}

\section{Abstract Result}\label{se:main}

Recall that $x^{\delta}$ and $x_{\eta}$ denote the solutions of the Eulerian~\eqref{eq:EulerianSol} and Lagrangian~\eqref{eq:xetaProj} formulation, respectively. \Cref{le:EulerLagrange} shows that the curves $(x^{\delta})_{\delta_{\min}\leq \delta\leq \delta_{\max}}$ and $(x_{\eta})_{\eta\geq0}$ are reparametrizations of one another. In particular, they trace out the same trajectory, and only the trajectory matters for the subsequent definition.

\begin{definition}\label{de:monotone}
   Let $\cP\subset\RR^{d}$ be a polytope and $c\in\RR^{d}$. We say that $\cP$ is \emph{$c$-monotone} if for %
   any face $F$ of $\cP$,
   \begin{align}
     x^{\delta}\in F \quad&\implies\quad x^{\delta'}\in F \quad \text{for all }\delta' \geq \delta\text{ in } [\delta_{\min},\delta_{\max}], \qquad \text{or equivalently if} \label{eq:monotone}\\
     x_{\eta} \in F \quad&\implies\quad x_{\eta'}\in F \quad \text{for all }\eta' \geq \eta\text{ in } [0,\infty]. \label{eq:monotoneLagrange}
   \end{align}   
   We say that $\cP$ is \emph{monotone} if it is $c$-monotone for all $c\in\RR^{d}$.
\end{definition}

This definition means that the ``support'' of $x^{\delta}$ is monotonically decreasing in~$\delta$, in the following sense. For any $x\in\cP$, there is a unique face $F=F(x)$ such that  $x\in\relint F$; moreover, $x\in\relint F$ if and only if $x$ is a convex combination of the vertices $\ext F$ with strictly positive weights \cite[Exercise 3.1, Theorem~5.6]{Brondsted.83}. Thus $\ext F(x)$ is a notion of support for~$x$, %
and then the property~\eqref{eq:monotone} indeed means that the support of $x^{\delta}$ is monotonically decreasing for inclusion. When $\cP$ is the unit simplex $\{x: x_i\geq0, \,\sum_i x_i=1\}$, then $\ext F(x)$ boils down to the usual measure-theoretic notion of support, namely $\{i:\, x_{i}>0\}$ for $x=(x_{1},\dots,x_{d})\in\cP$. This identification breaks down when~$\cP$ is the Birkhoff polytope, but monotonicity is nevertheless equivalent for the two notions of support (see \cref{le:BirkhoffSptMonEquiv}).\footnote{Clearly this equivalence does not hold in general. E.g., when $\cP\subset(0,1)^{d}$, all points have the same measure-theoretic support.} 
Next, we characterize monotonicity in geometric terms.

\begin{theorem}\label{th:main}
  A polytope $\cP$ is monotone (\cref{de:monotone}) if and only if
	\begin{myenum}
    \item\label{it:cond1}  $\proj_{\cP}(0)\in \relint \cP $ and
    \item \label{it:cond2} $\proj_{\aff F}(0)\in F$ for each face $\emptyset\neq F\subset\cP$.
  \end{myenum}  
\end{theorem}

The two conditions are similar, with the requirement for the proper faces $F$ being less stringent than for the improper face $\cP$: while $\proj_{F}(0)\in\relint F$ is a sufficient condition for~\labelcref{it:cond2}, the condition \labelcref{it:cond2} includes the boundary case where  $\proj_{F}(0)=\proj_{\aff F}(0)\in \rbd F$. %
We shall see that \labelcref{it:cond1} and \labelcref{it:cond2} play different roles in the proof. Simple examples show that changing $\relint \cP$ to~$\cP$ in \labelcref{it:cond1} or~$F$ to~$\relint F$ in \labelcref{it:cond2} invalidates the equivalence in \cref{th:main}. 

\subsection{Proof of \cref{th:main}}

We first prove the ``if'' implication, using the Lagrangian formulation (\cref{le:Lagrange}). This parametrization is convenient because $\eta\mapsto x_{\eta}$ is piecewise affine. While the latter is well-known (even for some more general norms, see \cite{FinzelLi.00} and the references therein); we detail the statement for completeness.

\begin{lemma}\label{le:affine}
  Let~$\cP \subseteq \R^{d}$ be a polytope, $c \in \R^{d}$ and let $x_{\eta}$ be the solution of \eqref{eq:LagrangeIntro} with $\eta > 0$. The curve $[0,\infty]\ni\eta\mapsto x_{\eta}$ is piecewise affine. The affine pieces correspond to faces of~$\cP$ as follows. Fix $\eta_{0}\in [0,\infty]$ and let $F$ be the unique face of $\cP$ such that $x_{\eta_{0}}\in \relint F$. %
  Let $I=\{\eta: x_{\eta}\in \relint F\}$. Then $I$ is an interval containing~$\eta_{0}$ and $\overline{I}\ni\eta\mapsto x_{\eta}\in F$ is affine. In particular, the curve $\eta\mapsto x_{\eta}$ does not return to $\relint F$ after leaving it.
\end{lemma}

\begin{proof}
  Consider $\eta_{1}<\eta_{2}$ in~$I$ and let $H= \aff F$. As $x_{\eta_{i}}\in \relint F$, we have 
  $x_{\eta_{i}}=\proj_{F}(-\eta_{i} c)=\proj_{H}(-\eta_{i} c)$. Since $H$ is an affine subspace, the curve $\eta\mapsto \proj_{H}(-\eta c)$ is affine. In particular, convexity of $\relint F$ implies that $\proj_{H}(-\eta c)\in \relint F$ for all $\eta\in[\eta_{1},\eta_{2}]$. Thus $I$ is an interval and $I\ni\eta\mapsto x_{\eta}\in \relint F$ is affine, and this extends to the closure by continuity.
\end{proof}

Recall that $\rbd K := K\setminus\relint K$ denotes the relative boundary of $K\subset\RR^{d}$. Consider the first time the curve $\eta\mapsto x_{\eta}$ touches a given face~$F$ of~$\cP$. That point may lie in $\relint F$ or in $\rbd F$. As seen in \cref{le:leavingFace} below, the boundary situation indicates that $\eta\mapsto x_{\eta}$ left another face $F_{*}$ when it entered $F\not\subset F_{*}$, meaning that $\cP$ is not monotone. Hence, we analyze that situation in detail in the next lemma.

\begin{lemma}\label{le:boundaryEntrance}
  Let~$\cP \subseteq \R^{d}$ be a polytope, $c \in \R^{d}$ and let $x_{\eta}$ be the solution of \eqref{eq:LagrangeIntro} with $\eta > 0$. Let $F$ be a face of $\cP$ and $I=\{\eta: x_{\eta}\in \relint F\}\neq\emptyset$. 
  Let $\eta_{0} :=\inf I$. If $x_{\eta_{0}}\in\rbd F$, then exactly one of the following holds true:
  \begin{enumerate}
  \item $\proj_{\aff F} (0)\notin F$;
  \item $\eta_{0}=0$, and then $x_{\eta_{0}}=x_{0}=\proj_{\cP} (0)=\proj_{\aff F} (0)\in \rbd F$. In particular, $x_{0}\notin\relint\cP$.
  \end{enumerate} 
\end{lemma}

\begin{proof}
  Recall from \cref {le:affine} that $I$ is an interval and $\overline{I}\ni\eta\mapsto x_{\eta}\in F$ is affine. Let $H= \aff F$ and consider $z_{\eta}:=\proj_{H}(-\eta c)$. Recalling \cref{Lemma:projset}, we have for $\eta\in I$ that $x_{\eta}= \proj_{\cP} (-\eta c)=\proj_{F} (-\eta c)=\proj_{H} (-\eta c)=z_{\eta}$ and in particular $z_{\eta}\in \relint F$. 
	
	As $x_{\eta}=z_{\eta}$ for $\eta\in I$ and both curves are continuous, it follows that $x_{\eta_{0}}=z_{\eta_{0}}$. Therefore, our assumption that $x_{\eta_{0}}\in\rbd F$ implies that  $z_{\eta_{0}}\in\rbd F$. 

  Since $H$ is an affine space, $\RR_{+}\ni \eta\mapsto z_{\eta}$ is affine. We have $z_{\eta_{0}}\in\rbd F$ and $z_{\eta}\in \relint F$ for all $\eta\in (\eta_{0},\eta_{1})$, where $\eta_{1}:=\sup I > \eta_{0}$. (Note that~$I$ cannot be a singleton when $z_{\eta_{0}}\in\rbd F$.) As $F$ is convex, it follows that $z_{\eta}\notin F$ for $0\leq \eta <\eta_{0}$. If $z_{0}\notin F$, we are in the first case. Whereas if $z_{0}\in F$, it follows that $\eta_{0}=0$. Thus $x_{\eta_{0}}=x_{0}=\proj_{\cP}(0)=\proj_{F}(0)$. Moreover, $z_{\eta_{0}}=x_{\eta_{0}}\in F$ implies $x_{\eta_{0}}=\proj_{\aff F}(0)$.
\end{proof} 

\begin{lemma}\label{le:leavingFace}
  Let~$\cP \subseteq \R^{d}$ be a polytope, $c \in \R^{d}$ and let $x_{\eta}$ be the solution of \eqref{eq:LagrangeIntro} with $\eta > 0$. Let $\eta_{*}\geq0$ and let $F_{*}$ be the unique face of $\cP$ with $x_{\eta_{*}}\in\relint F_{*}$. Suppose there exists $\eta>\eta_{*}$ such that $x_{\eta}\notin F_{*}$ and let $\eta_{0}=\max \{\eta\geq \eta_{*}:\, x_{\eta}\in F_{*}\}$. For sufficiently small $\eta_{1}>\eta_{0}$, there is a unique face $F$ such that $x_{\eta}\in\relint F$ for all $\eta \in (\eta_{0},\eta_{1})$.\footnote{When traveling along the curve $(x_{\eta})_{\eta\geq\eta_{*}}$, $F$ is the first face outside $F_{*}$ whose interior is visited.} If $x_{0}:=\proj_{\cP} (0)\in\relint\cP$, then~$F$ satisfies $\proj_{\aff F} (0)\notin F$.
\end{lemma}

\begin{proof}
  For $x\in\cP$, let $F(x)$ denote the unique face such that $x\in \relint F(x)$. By \cref{le:affine}, the map $(\eta_{0},\infty)\ni\eta\mapsto  F(x_{\eta})$ has finitely many values and the preimage of each value is an interval. Hence the map must be constant on $(\eta_{0},\eta_{1})$ for sufficiently small $\eta_{1}>\eta_{0}$. Let $F=F(x_{\eta})$, $\eta\in (\eta_{0},\eta_{1})$ be the corresponding face. As $F_{*}$ and $F$ are faces with $F_{*}\not\subset F$, we have $F\cap F_{*}\subset\rbd F$. Thus $x_{\eta_{0}}\in F\cap F_{*}$ implies $x_{\eta_{0}}\in\rbd F$, and now \cref{le:boundaryEntrance} applies.
\end{proof}

\begin{proof}[Proof of \cref{th:main}]
  \emph{Step 1: \labelcref{it:cond1}, \labelcref{it:cond2} $\Rightarrow$ monotone}. If $\cP$ is not monotone, then \cref{le:leavingFace} shows that either  \labelcref{it:cond1} or \labelcref{it:cond2} must be violated.
  
  \emph{Step 2: Not \labelcref{it:cond1} $\Rightarrow$ not monotone}. Suppose that $x_0:=\proj_{\cP} (0) \notin \relint \cP$. 
As $x_{0} \in  \cP\setminus (\relint \cP)$, there exists a hyperplane $H= \left\{ x \in \RR^{d} \text{ : } \langle a, x \rangle = b \right\}$ with 
$$
x_0\in H, \qquad  \cP \subseteq  \left\{ x \in \RR^{d} \text{ : } \langle a, x \rangle \leq b \right\}, \qquad  \relint \cP \subseteq  \left\{ x \in \RR^{d} \text{ : } \langle a, x \rangle < b \right\}.
$$
Define $c=a$ and $F=H\cap \cP$ and $\delta_{0} = \|x_0\|$, and denote $\overline{B}_{\delta} =\{ x\in \RR^d: \ \|x\|\leq \delta\}$. Then $\cP \cap \overline{B}_{\delta_0} = \left\{ x_{0} \right\}$ and in particular $x^{\delta_{0}} = x_{0}$. Moreover, 
\begin{equation}\label{eq:Not(i)NotMon}
     \langle c,x \rangle=b\quad \text{for all } x\in F,\quad  \text{while}\quad 
    \langle c, x\rangle <  b \quad \text{for all } x\in \relint \cP.
\end{equation}
For any $\delta >\delta_{0}$, we have $\overline{B}_\delta \cap \relint \cP\neq \emptyset$, and then~\eqref{eq:Not(i)NotMon} implies $x^{\delta} \notin F$. In particular, $\cP$ is not monotone for the cost $c$.
  
  \emph{Step 3: Not \labelcref{it:cond2} $\Rightarrow$ not monotone}. Given the assumption,  there exists a face $F$ of $\cP$ such that $\proj_{\aff F} (0)$ does not belong to $F$. Set  $x_{\min}:=\proj_{F} (0)$. 
As $F$ is a face of $\cP$, there is a hyperplane $H=\{ x\in \RR^d: \ \langle a, x\rangle=b \}$ such that 
$$ F= H\cap \cP \quad {\rm and}\quad  \cP\subset \{ x\in \RR^d: \ \langle a, x\rangle \geq b \}.  $$
Define $c=a-x_{\min}$ and $\delta_1= \|x_{\min}\|$. 
Then for every  $x\in \cP\cap \overline{B}_{\delta_1}$,
$$ \langle c,x \rangle= \langle a, x \rangle-\langle x_{\min}, x\rangle \geq  b-\| x\|\|x_{\min}\|\geq b-\| x_{\min}\|^2=\langle c, x_{\min}\rangle,$$
showing that $x^{\delta_1}=x_{\min}$. 
Recall that $x_{\min}:=\proj_{F} (0)$, so that $\langle x_{\min}, x-x_{\min} \rangle \geq 0$ for all $x\in F$ (\cref{Lemma:projset}). Thus $x_{\min}\neq \proj_{\aff F} (0)$ implies that
$$ F':= \{ x\in \RR^d: \ \langle x_{\min}, x-x_{\min} \rangle=0\}\cap F  $$
is a proper face of $F$---and {\it a fortiori} a face of $\cP$---such that 
$$ \emptyset \neq F\setminus F' \subset  \{ x\in \RR^d: \ \langle x_{\min}, x-x_{\min} \rangle>0\}. $$
We claim that $x^{\delta}\notin F'$ for all $\delta> \delta_1$. Indeed, 
for any $x'\in F'$
it holds that
$$ \langle c, x' \rangle= \langle a, x' \rangle-\langle x_{\min}, x'\rangle =  b-\langle x_{\min}, x'\rangle= \langle c,  x_{\min}\rangle, $$
whereas for any $x\in F\setminus F'$, it holds that
$$ \langle c,x \rangle= b-\langle x_{\min}, x\rangle =  \langle c, x_{\min} \rangle-\langle x_{\min}, x-x_{\min}\rangle< \langle c, x_{\min} \rangle.$$
In summary, $\langle c, x \rangle < \langle c, x' \rangle$ for all $x\in F\setminus F'$ and $x'\in F'$. In view of $x^{\delta_{1}}=x_{\min}\in F'$, it follows that $x^{\delta}\notin F'$ for all $\delta> \delta_1$ and in particular that $\cP$ is not monotone for the cost~$c$.
\end{proof} 

\section{Applications}\label{se:applications}

\subsection{Soft-Min}

To begin with a straightforward example, let $\cP=\Delta$ be the unit simplex in Euclidean space~$\RR^{d}$ and $c=(c_{1},\dots,c_{d})\in\RR^{d}$. Then
$$
  \min_{x\in\Delta} \langle c,x \rangle = \min_{1\leq i\leq d} c_{i}
$$
corresponds to finding the minimum value of $c$. More specifically, $x=(x_{1},\dots,x_{d})$ is a minimizer if and only if $x$ is supported in $\argmin c$; i.e., $\spt x=\{i:\, x_{i}\neq0\}\subset \argmin c$. When this linear program is regularized with the entropy of $x$, we obtain the usual soft-min (counterpart of log-sum-exp) which gives large weights to the small values of $c$ but non-zero weights to all values. The quadratic regularization,
\begin{align}\label{eq:quadSoftMin}
  x^{\delta}=\argmin_{x\in\Delta:\, \|x\|\leq\delta} \langle c,x \rangle \qquad \text{or}\qquad
  x_{\eta}=\argmin_{x\in\Delta} \langle c,x \rangle + \frac{1}{2\eta}\|x\|^{2},
\end{align} 
yields a sparse soft-min: as $\delta\to\delta_{\max}$ (or as $\eta\to\infty$), the support of the solution tends to $\argmin c$. In this example,  $\delta_{\min} = \frac{1}{\sqrt{d}}$ and $\delta_{\max} = \frac{1}{\sqrt{ \#\argmin c}}$.

\begin{corollary}\label{co:simplex}
  The unit simplex $\Delta\subset\RR^{d}$ is monotone for all $d\geq1$; i.e., the support of the soft-min $x^{\delta}$ (or $x_{\eta}$) defined in~\eqref{eq:quadSoftMin} decreases monotonically to $\argmin c$ as $\delta\in[\delta_{\min},\delta_{\max}]$  (or $\eta\geq0$) increases.
\end{corollary} 

\begin{proof}
We have $\ext \Delta = \{e_{1}, \dots, e_{d}\}$, the canonical basis of~$\RR^{d}$.  Let $F \subset\Delta$ be a face, then there exist $\{i_{1}, \dots, i_{k}\} \subseteq \{1, \dots, d\}$ such that $F = \coh(e_{i_{1}},\dots, e_{i_{k}})$. Note that $\proj_{F}(0)= \sum_{j=1}^{k}\lambda_{j}e_{i_{j}}$, where $\lambda=(\lambda_{1},\dots,\lambda_{k})$ is the solution of 
\begin{align*}
    \min_{\lambda \in \RR^{d}} \big\| \sum_{j=1}^{k}\lambda_{j}e_{i_{j}} \big\|^{2} \quad \text{ s.t. } \sum_{j=1}^{k} \lambda_{j} = 1,
\end{align*}
namely $\lambda= (k^{-1}, \dots, k^{-1})$. In particular, $\proj_{F}(0)=k^{-1}\sum_{j=1}^{k} e_{i_{j}}\in\relint F$. Now \cref{th:main} applies.
\end{proof} 

\subsection{Optimal Transport}

We recall the quadratically regularized optimal transport problem
\begin{equation}
    \tag{QOT}\label{EQOT}
    \inf_{\gamma\in \Gamma(\mu, \nu)} \int \hat c(x,y)d\gamma(x,y)+\frac{1}{2\eta}  \left\| \frac{d\gamma}{  d(\mu \otimes \nu) }\right\|_{L^2( \mu \otimes \nu)}^2
\end{equation} 
where $\Gamma(\mu, \nu)$ denotes the set of couplings of the probability measures $(\mu, \nu)$. We will be concerned with the discrete case as introduced in \cite{blondel18quadratic, DesseinPapadakisRouas.18, EssidSolomon.18}. See also \cite{LorenzMahler.22, LorenzMannsMeyer.21, Nutz.24} for more general theory, \cite{BayraktarEcksteinZhang.22, EcksteinNutz.22, GarrizmolinaEtAl.24} for asymptotic aspects, \cite{EcksteinKupper.21, GeneveyEtAl.16, GulrajaniAhmedArjovskyDumoulinCourville.17, LiGenevayYurochkinSolomon.20, seguy2018large} for computational approaches and~\cite{LiGenevayYurochkinSolomon.20, ZhangMordantMatsumotoSchiebinger.23} for some recent applications.

Fix $N\in\NN$ and two sets of distinct points, $\{X_{1},\dots,X_{N}\}\subset\RR^{D}$ and $\{Y_{1},\dots,Y_{N}\}\subset\RR^{D}$. Let $\mu = \frac{1}{N} \sum_{i=1}^{N} \delta_{X_{i}}$ and $\nu = \frac{1}{N} \sum_{i=1}^{N} \delta_{Y_{i}}$ denote the associated empirical measures, and let $\hat c:\RR^{D}\times\RR^{D}\to\RR$ be a (cost) function. We note that $\mu \otimes \nu=\frac{1}{N^2}\sum_{i,j=1}^N \delta_{(X_i, Y_j)}$ while $d\gamma/d(\mu \otimes \nu)$ is simply the ratio of the probability mass functions. Any coupling  $\gamma \in \Gamma(\mu, \nu) $ can be identified with its matrix of probability weights 
$$
(\gamma_{i,j})_{i,j=1}^n\in \Gamma_{N}=\left\{\gamma\in\RR^{N\times N}:\, \sum_{i=1}^N \gamma_{i,j}  = \frac{1}{N},   \; \sum_{j=1}^N \gamma_{i,j}=\frac{1}{N}, \; \gamma_{i,j}\geq 0\right\}
$$
via $\gamma = \sum_{i,j=1}^N \gamma_{i,j}\delta_{(X_i, Y_j)}$. Writing similarly $c_{ij}=\hat c(X_{i},X_{j})$, \eqref{EQOT} can be identified with the problem
\[
 \inf_{\gamma\in \Gamma_N} \langle c, \gamma \rangle +\frac{\eps N}{2}  \|\gamma\|^2. 
\]
Clearly $\Gamma_{N}=\frac{1}{N}\Pi_{N}$ where  
$$
  \Pi_{N}=\left\{\pi\in\RR^{N\times N}:\, \sum_{i=1}^N \pi_{i,j}  = 1,   \; \sum_{j=1}^N \pi_{i,j}=1, \; \pi_{i,j}\geq 0\right\}
$$ 
denotes the {\it Birkhoff polytope} of doubly stochastic matrices; see, e.g., \cite{Brualdi.06}. 
The monotonicity (in the sense of \cref{de:monotone}) of $\Pi_N$ is equivalent to that of $\Gamma_N$. 

The following result connects the measure-theoretic support $\spt \pi =\{(i,j): \, \pi_{ij}>0\}$ with the geometric notion discussed below \cref{de:monotone}.

\begin{lemma}\label{le:BirkhoffSptMonEquiv}
    Let $\pi, \pi'\in \Pi_N$. Denote by $F(\pi)$ the unique face $F\subset\Pi_N$ such that $\pi\in\relint F$. For $\pi, \pi'\in\Pi_N$, the following are equivalent:
  \begin{enumerate}
    \item $\spt \pi' \subset \spt \pi$,
    \item $F(\pi')\subset F(\pi)$,
    \item if $\pi\in F$ for some face $F\subset\Pi_N$, then $\pi'\in F$.
  \end{enumerate}
\end{lemma}

\begin{proof}
  The equivalence of (ii) and (iii) follows from the fact that $F(\pi)$ is the smallest face containing~$\pi$ \cite[Theorem~5.6]{Brondsted.83}. The equivalence of (i) and (iii) is deferred to \cref{lemma:zerosIffMonot} below. (The implication (ii) $\Rightarrow$ (i) is straightforward, and holds for any polytope~$\cP$ of measures.)
\end{proof} 

As a consequence of \cref{le:BirkhoffSptMonEquiv}, we have the following equivalence.

\begin{corollary}\label{co:monotoneEquivalent}
    Let $\gamma_{\eta,\hat c}$ denote the solution of \eqref{EQOT} for $\hat c: \RR^N\times \RR^N\to \RR$. The Birkhoff polytope is monotone if and only if the optimal support $\spt \gamma_{\eta,\hat c} $ is monotone decreasing in~$\eta$ for all~$\hat c$.
\end{corollary}

\begin{theorem}\label{th:BirkhoffMon}
  The Birkhoff polytope $\Pi_{N}$ is monotone for $1\leq N\leq 4$ but not monotone for $N\geq 5$. In particular, when $N\geq5$, the optimal support $\eta\mapsto\spt(\gamma_{\eta,\hat c})$ is not monotone for some costs $\hat c$.
\end{theorem} 

\begin{example}\label{ex:counterexCost}
We exhibit an example for $N=5$ where the support is not monotone. (This is not easily achieved by brute-force numerical experiment). Our starting point is the matrix~$A$ in~\eqref{eq:badA} below, which is used to describe a face where \labelcref{it:cond2} fails. Step~3 in the proof of \cref{th:main} then suggests to perturb $-A$ in a suitable direction to find a cost~$c$ exhibiting non-monotonicity. With some geometric considerations, this leads us to propose the cost matrix
\[
c=\begin{bmatrix}
-1.1 & -1 & -1 & -1 & -1 \\
-1 & -1.1 & 0  & 0  & 0  \\
-1 & 0  & -1.1 & 0  & 0  \\
-1 & 0  & 0  & -1.1 & 0  \\
-1 & 0  & 0  & 0  & -1.1
\end{bmatrix}.
\]
(If one prefers a nonnegative cost matrix, one can obtain that by adding 1.1 to all entries. This translation does not change the solution $\gamma_{\eta}$ to~\eqref{EQOT}.) For $\eta=2.5$, or equivalently regularization strength $1/(2\eta)=0.2$, the corresponding problem~\eqref{EQOT} has an exact solution coupling $\gamma_{\eta}$ with probability weights given by
\[(\gamma_{\eta=2.5})=
\begin{bmatrix}
0  & 0.05 & 0.05 & 0.05 & 0.05 \\
0.05 & 0.15 & 0  & 0  & 0  \\
0.05 & 0  & 0.15 & 0  & 0  \\
0.05 & 0  & 0  & 0.15 & 0  \\
0.05 & 0  & 0  & 0  & 0.15
\end{bmatrix}.
\]
(This can be verified by noticing that the stated coupling is induced by the dual potentials $f=g=(-0.575, -0.175, -0.175, -0.175, -0.175)$; see for instance \cite[Theorem~2.2 and Remark~2.3]{Nutz.24}.) We observe in particular that the location $(X_{1},Y_{1})$ is not in the support. On the other hand, because the diagonal of $c$ features the smallest costs, the solution $\gamma_{\infty}$ of the unregularized transport problem is to put all mass on the diagonal; i.e., to transport all mass from $X_{i}$ to $Y_{i}$ for each~$i$. Because of the stationary convergence mentioned in \cref{le:Lagrange}, that is also the solution of the regularized problem for large enough value of $\eta$ (e.g., $\eta=100$ will do):
\[
(\gamma_{\eta=100})=(\gamma_{\infty}) = 
\begin{bmatrix}
0.2 & 0 & 0 & 0 & 0 \\
0 & 0.2 & 0 & 0 & 0 \\
0 & 0 & 0.2 & 0 & 0 \\
0 & 0 & 0 & 0.2 & 0 \\
0 & 0 & 0 & 0 & 0.2
\end{bmatrix}.
\]
In particular, $(X_{1},Y_{1})$ is part of the support for large $\eta$, completing the example. \Cref{fig:counterexample_weights} shows in more detail the weight at $(X_{1},Y_{1})$ as a function of (the inverse of) $\eta$.
\end{example}

\begin{figure}[h!] 
		\centering
		\resizebox{1\linewidth}{!}{%
		\includegraphics{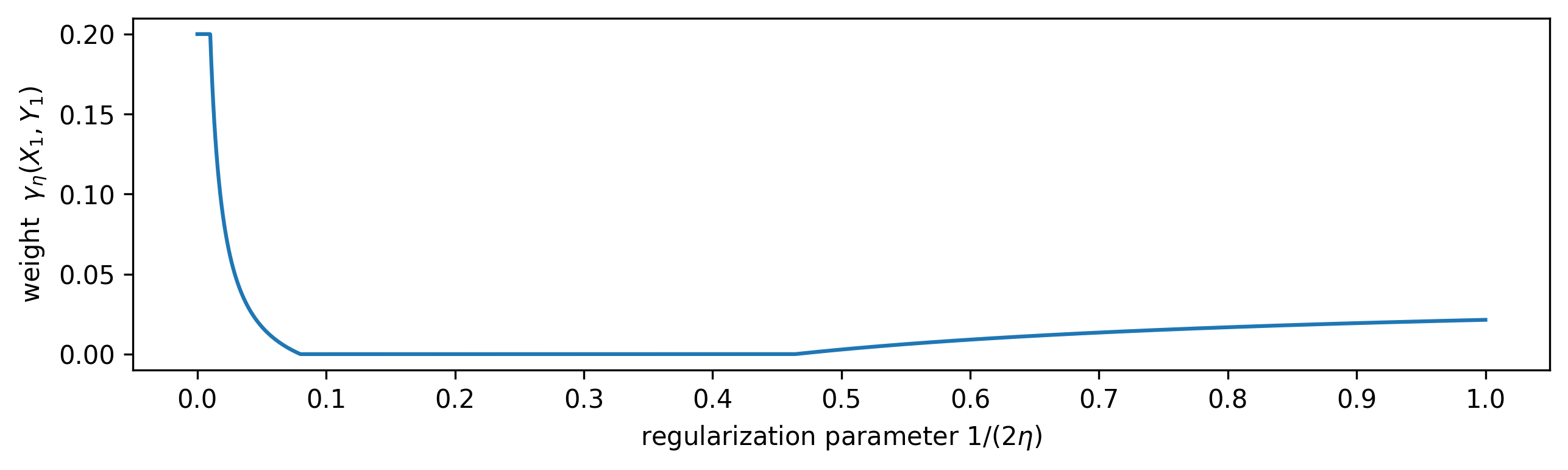}
		
		}%

\caption{Probability mass $\gamma_{\eta}(X_{1},Y_{1})$ plotted against $1/(2\eta)$, showing that $(X_{1},Y_{1})$ is in the support for small and large values of $\eta$ but outside for an intermediate interval.}
\label{fig:counterexample_weights}
\end{figure}

\begin{remark}\label{rk:EssidSolomon}
  The early work of \cite{EssidSolomon.18} considered a minimum-cost flow problem with quadratic regularization that predates the optimal transport literature for this regularization. The authors point out that the solution can have non-monotone support already in the minimal setting of $2\times2$ points~\cite[Figure~1]{EssidSolomon.18}. Similarly, it is not hard to obtain non-monotonicity if instead of $\mu\otimes\nu$ we use a different measure to define the $L^{2}$-norm in~\eqref{EQOT}, similarly as in~\cite[Example~A.1]{Nutz.24}. In those examples, the mechanism causing non-monotonicity is straightforward and quite different than in the present work, where non-monotonicity arises only in higher dimensions.
\end{remark}

\subsection{Proof of \cref{le:BirkhoffSptMonEquiv} and \cref{th:BirkhoffMon}}

The celebrated  Birkhoff's theorem \cite{Birkhoff.46} (or, e.g., \cite[p.\,30]{MarshallOlkinArnold.11}) states that the vertices $\ext \Pi_N$ are the permutation matrices of $\{1,\dots,N\}$; that is, the elements of $\Pi_N$ with binary entries. Following \cite{Brualdi.77}, the faces of $\Pi_N$  can be described using the so-called permanent function. If $A$ is a binary $N\times N$ matrix, its \emph{permanent} ${\rm per}(A)\in\NN$ is defined as the number of permutation matrices $P$ with $P\leq A$ (meaning that $P_{i,j}\leq A_{i,j}$ for all $i,j$). Denoting by $\coh(\cdot)$ the convex hull, the following characterization is contained in~\cite[Theorem~2.1]{Brualdi.77}.

\begin{lemma}\label{Lemma:Brualdi}
    Let $t\in\NN$ and $P^{(1)}, \dots, P^{(t)}\in \ext(\Pi_N)$. Let $A=(A_{i,j})_{i,j=1}^N$ be the matrix such that $A_{i,j}=1$ if there exists $s\in \{1, \dots, t\}$ with $P^{(s)}_{i,j}=1$ and $A_{i,j}=0$ otherwise. Then $\coh(\{P^{(1)}, \dots, P^{(t)}\})$ is a face of $\Pi_N$ if and only if  ${\rm per}(A)=t$. 
\end{lemma}
We use \cref{Lemma:Brualdi} to relate faces with the distribution of zeros. 

\begin{lemma}\label{lemma:zerosIffMonot}
    Let $\pi, \pi'\in \Pi_N$. The following are equivalent:
    \begin{enumerate}
        \item there exists $(i,j)$  such that $\pi_{i,j}=0 $ and $\pi_{i,j}'>0 $; 
        \item there exists a face $F$ of $\Pi_N$ such that $\pi\in F  $ and $\pi'\notin F$. 
    \end{enumerate}
\end{lemma}

\begin{proof}
	{\it (i) $\Rightarrow$ (ii):}  Let $(i,j)$ be such that $\pi_{i,j}=0 $ and $\pi_{i,j}'>0 $. Assume w.l.o.g.\ that $(i,j)=(1,1)$.  Then $\pi $ belongs to the set 
    $$F=\{ \pi\in \Pi_N: \langle \pi, A \rangle=N\},$$
    where $A$ is the matrix 
\[
A = \begin{pmatrix}
0 & 1 & 1 & \cdots & 1 \\
1 & 1 & 1 & \cdots & 1 \\
1 & 1 & 1 & \cdots & 1 \\
\vdots & \vdots & \vdots & \ddots & \vdots \\
1 & 1 & 1 & \cdots & 1
\end{pmatrix}. 
\]
As
$ \Pi_N\subset \{ \pi\in \Pi_N: \langle \pi, A \rangle\leq N\},  $ we see that~$F$ is a face. Clearly $\pi'$ does not belong to $F$, proving the claim.

  {\it (ii) $\Rightarrow$ (i):}
Let $P^{(1)}, \dots, P^{(t)}\in \ext \Pi_N$ be such that $ F=\coh(\{P^{(1)}, \dots, P^{(t)}\})$ and let $A=(A_{i,j})_{i,j=1}^N$ be the matrix such that $A_{i,j}=1$ if there exists $s\in \{1, \dots, t\}$ such that $P^{(s)}_{i,j}=1$ and $A_{i,j}=0$ otherwise. We have ${\rm per}(A)=t$ by \cref{Lemma:Brualdi}. As an element of $\Pi_N$, $\pi'$ can be written as a convex combination of permutation matrices, i.e., $$\pi'=\sum_{P\in \ext (\Pi_N) } \lambda_P P, \quad {\rm with}\quad \lambda_P\in [0,1] \ {\rm and}  \ \sum_{P\in \ext (\Pi_N) } \lambda_P=1.$$ 
As $\pi'\notin F$, there exists  $P\in \ext(\Pi_N) \setminus \{P^{(1)}, \dots, P^{(t)}\} $ with $\lambda_P>0$. Using the fact that  ${\rm per}(A)=t$, we derive the existence of $(i,j) $ such that  $P_{i,j}=1$ but $A_{i,j}=0$. In particular, $\pi_{i,j}'\geq \lambda_P>0$ but $\pi_{i,j}=0$, proving the claim.
\end{proof}

\begin{lemma}\label{le:specialFace}
  Let $N\geq1$. Define the permutation matrices $P^{k}\in\R^{N\times N}$, $1\leq k \leq N$ by 
\begin{align*}
    P^{k}_{i,j} = \begin{cases} 1 & \text{ if } i = 1 \text{ and } j = k, \\
        1 & \text{ if } i = k \text{ and } j = 1, \\
        1 & \text{ if } 1 \neq i = j \neq k, \\
        0 & \text{ otherwise.}
            \end{cases}
\end{align*}
That is, $P^{k}$ permutes the first with the $k$-th element. Then $F := \coh(P^{1}, \dots, P^{N})$ is a face of~$\Pi_{N}$ and 
\begin{align}\label{eq:specialLambda}
  \proj_{\aff F}(0)=\sum_{i=1}^{N} \lambda_{k} P^{k} \quad\text{for}\quad \lambda_{1} = \frac{4-N}{N+2}, \quad \lambda_{2}=\cdots=\lambda_{N} = \frac{2}{N+2}.
\end{align} 
As a consequence, 
\begin{align*}
  \proj_{\aff F}(0)\in \begin{cases}
   \relint F, & \text{ if } N=1,2,3, \\
   \rbd F, & \text{ if } N=4, \\
   \R^{N\times N}\setminus F& \text{ if } N\geq 5,
  \end{cases} 
\end{align*}
and in particular \labelcref{it:cond2} is violated for $N\geq5$.
\end{lemma}

\begin{proof}
Define $A \in \RR^{N\times N}$ by
\begin{align}\label{eq:badA}
    A_{i,j} = \begin{cases} 1 & \text{ if } i = j,\\
        1 & \text{ if } i = 1 \text{ or } j = 1, \\
        0 & \text{ otherwise,}
        \end{cases}
\end{align}
so that $A$ is the entry-wise maximum $A=\max(P^{1}, \dots, P^{N})$. One readily verifies that ${\rm per}(A)=N$. Hence by \cref{Lemma:Brualdi}, $F := \coh(P^{1}, \dots, P^{N})$ is a face of~$\Pi_{N}$. To determine $\proj_{\aff F}(0)$, we consider the minimization problem
\begin{align*}
    \min_{\lambda=(\lambda_{1},\dots,\lambda_{N}) \in \RR^{N}} \big\| \sum_{i=1}^{N}\lambda_{i}P^{i} \big\|^{2} \quad \text{ s.t.\ } \quad\sum_{i=1}^{N} \lambda_{i} = 1,
\end{align*}
where $\|\cdot\|$ denotes the Frobenius norm. The Lagrangian for this problem is
\begin{align*}
    L(\lambda, \rho) = \lambda_{1}^{2} + 2\sum_{j = 2}^{N}\lambda_{j}^{2} + \sum_{j=2}^{N}\Big( \sum_{k \neq j} \lambda_{k} \Big)^{2} + \rho \Big( 1 - \sum_{j = 1}^{N} \lambda_{j} \Big).
\end{align*}
Here the first three terms arise from the objective: the first is the value $\big(\sum_{i=1}^{N} \lambda_{i}P^{i} \big)_{1,1}$, the second is the sum of all remaining terms in the first row and column, and the third is the sum of the remaining terms in the diagonal. Finally, the fourth term arises from the constraint. The resulting optimality equations are
\begin{align*}
    \rho & = 2N\lambda_{1} + 2(N-2) \sum_{j=2}^{N} \lambda_{j}, \qquad \sum_{j=1}^{N} \lambda_{j} = 1, \\
    \rho & = 2N\lambda_{i} + 2(N-2)\lambda_{1} +  2(N-3)\sum_{\substack{j=2 \\ j \neq i}}^{N} \lambda_{j} \quad \text{ for } 2 \leq i \leq N.
\end{align*}
By symmetry, the unique optimal $\lambda$ satisfies $\lambda_{i} = \lambda_{j} =: \lambda_{0}$ for all $i,j \geq 2$, so that
\begin{align*}
    & 2N\lambda_{0} + 2(N-2)\lambda_{1}  +  2(N-2)(N-3) \lambda_{0} = 2N\lambda_{1} + 2(N-1)(N-2) \lambda_{0}, \\ 
    & \lambda_{1} + (N-1)\lambda_{0} = 1
\end{align*}
and finally
\begin{align*}
    \lambda_{1} = \lambda_{0}\frac{(4 - N)}{2}, \qquad \lambda_{1} + (N-1)\lambda_{0} = 1.
\end{align*}
Solving for $\lambda_{1}$ and $\lambda_{0}$ yields~\eqref{eq:specialLambda}. Note that $\lambda_{0}>0$ for any $N$ whereas $\lambda_{1}>0$ for $N\in\{1,2,3\}$, $\lambda_{1}=0$ for $N=4$ and $\lambda_{1}<0$ for $N\geq5$, implying the last claim.
\end{proof}

The matrix $A$ of~\eqref{eq:badA} is inspired by the $4\times 4$ matrix in \cite[Example~1.5]{BrualdiDahl.22} where the authors are interested in a different problem (and, in turn, credit \cite{Achilles.78}).\footnote{In \cite{BrualdiDahl.22}, the conclusion is that ``not every zero pattern of a fully indecomposable $(0, 1)$-matrix is realizable as the zero pattern of a doubly stochastic matrix whose diagonal sums avoiding the $0$'s are constant.''} For the present question of monotonicity, this matrix yields a counterexample only for $N\geq5$. As an aside, the subsequent proof that for $N=4$, the face~$F$ considered in \cref{le:specialFace} is (up to symmetries) the only face where $\proj_{\aff F}(0)\in \rbd F$, whereas all other faces $F'$ satisfy $\proj_{\aff F'}(0)\in \relint F'$.

\begin{lemma}\label{le:BirkhoffMonotone}
  Let $1\leq N\leq4$. Then $\Pi_{N}$ is monotone.
\end{lemma} 

\begin{proof}
  We verify \labelcref{it:cond1} and \labelcref{it:cond2}. Note that $\proj_{\Pi_{N}}(0)$ is the matrix with all entries equal to $1/N$ (corresponding to the product measure). This matrix is clearly in the relative interior of~$\Pi_{N}$, showing \labelcref{it:cond1}. The property \labelcref{it:cond2} is trivial for $N=1,2$. For $N=3$ and $N=4$, we give a computer-assisted proof in the interest of brevity.\footnote{An analytic proof is also available, but requires us to go through 52 different cases for $N=4$.}\footnote{The cases $N\leq3$ can also be obtained as a corollary of the case $N=4$. One can check directly that if the Birkhoff polytope $\Pi_{N}$ is (not) monotone for some $N\in\NN$, then $\Pi_{N'}$ is also (not) monotone for all $N'\leq (\geq)N$.} Specifically, we generate all $N\times N$ permutation matrices and determine all families $\{P_{1}, \dots, P_{m}\}$ of permutation matrices that form the vertices of a nonempty face~$F$ by using the permanent function as in \cref{Lemma:Brualdi}. There are 49 nonempty faces for $N=3$ and 7443 for $N=4$. For each face~$F$, we can numerically compute $\proj_{\aff F}(0)$ or more specifically scalar coefficients $(\lambda_{k})_{1\leq k\leq m}$ such that $\proj_{\aff F}(0)=\sum_{k} \lambda_{k} P_{k}$ and $\sum_{k} \lambda_{k} \leq 1$. The coefficients $\lambda_{k}$ are non-unique for some faces; in that case, we choose positive weights if possible. Note that as $P_{k}$ are binary matrices and $N\leq 4$, the computation can be done with accuracy close to machine precision. It turns out that for $N=3$, all coefficients satisfy $\lambda_{k}>0.01$ (much larger than machine precision), establishing that $\proj_{\aff F}(0)\in \relint F$ and in particular \labelcref{it:cond2}. For $N=4$, most faces have coefficients $\lambda_{k}>0.01$, whereas for 96 faces, one coefficient is numerically close to zero, hence requiring an analytic argument. We verify that all those 96 faces are equivalent up to permutations of rows and columns, corresponding to a relabeling of the points~$X_{i}$ and~$Y_{j}$. Specifically, they are all equivalent to the particular face analyzed in \cref{le:specialFace}, where we have seen that $\proj_{\aff F}(0)\in F$ for $N\leq 4$ (and in fact $\proj_{\aff F}(0)\in \rbd F$ for $N=4$). We conclude that \labelcref{it:cond2} holds for all faces $F$ when $N=4$, completing the proof.
\end{proof}

\begin{proof}[Proof of \cref{th:BirkhoffMon}]
  In view of \cref{th:main}, the last statement of \cref{le:specialFace} implies that $\Pi_{N}$ is not monotone for $N\geq 5$, whereas we have seen in \cref{le:BirkhoffMonotone}  that $\Pi_{N}$ is monotone for $1\leq N\leq4$. 
\end{proof}

\section{On a Problem of Erd{\H o}s}\label{se:Erd{\H o}s}\label{se:erd}

Let $A=(a_{ij})$ be an $N\times N$ matrix. Then its maximal trace is defined as 
$$
  \maxtr A := \max_\sigma \sum_{i=1}^N a_{i,\sigma(i)}
$$
where $\sigma$ ranges over all permutations of $\{1,\dots,N\}$. 
When $A\in\Pi_N$ (i.e., $A$ is doubly stochastic) a result of the early paper~\cite{MarkusMinc.65} states that 
\begin{equation}\label{eq:erdIn}
  \|A\|^{2} \leq \maxtr A,
\end{equation}
where $\|A\|^{2}=\sum_{i,j} a_{ij}^{2}$ is the squared Frobenius norm. For a simple proof of~\eqref{eq:erdIn}, note that the maximal trace can be expressed as
\begin{equation}\label{eq:erdLink}
  \maxtr A = \max_{P\in\ext \Pi_N} \langle P,A \rangle
\end{equation}
where $P$ ranges over the set $\ext \Pi_N$ of all $N\times N$ permutation matrices (Birkhoff's theorem) and the inner product is the Frobenius one, $\langle B,A \rangle =\sum_{i,j} b_{ij}a_{ij}$ for $A=(a_{ij})$ and $B=(b_{ij})$. Consider the linear optimization problem 
$
  \max_{B\in \Pi_{N}} \langle B,A \rangle.
$
Its maximum must be attained at an extreme point. On the other hand, $A\in\Pi_{N}$, hence
$$
  \|A\|^{2} = \langle A,A \rangle \leq \max_{B\in \Pi_{N}} \langle B,A \rangle = \max_{P\in\ext \Pi_N} \langle P,A \rangle = \maxtr A,
$$
proving~\eqref{eq:erdIn}. Erd{\H o}s posed the following problem; see also \cite{BouthatMashreghiMorneauGuerin.24} for further background.

\begin{question}[Erd{\H o}s]
For which doubly stochastic matrices $A$ is inequality~\eqref{eq:erdIn} saturated?
\end{question} 

In formulas, the problem is to describe the set $\{A\in\Pi_{N}: \, \|A\|^{2} = \maxtr A\}$. To simplify the discussion, let us introduce the following terminology. 

\begin{definition}\label{de:erd}
  We call $A\in\R^{N\times N}$ an \emph{Erd{\H o}s matrix} if it is doubly stochastic (i.e., $A\in\Pi_{N}$) and $\|A\|^{2} = \maxtr A$.
\end{definition} 

\begin{definition}\label{de:monFace}
  Let $F$ be a face of $\Pi_{N}$. We say that $F$ is \emph{centered} if $\proj_{\aff F} (0) \in F$. 
\end{definition} 

In this terminology, \cref{th:main} states that $\Pi_{N}$ is monotone iff every face of $\Pi_{N}$ is centered. The following theorem shows that Erd{\H o}s matrices can be parametrized by the centered faces of~$\Pi_{N}$; in particular, we provide a geometric answer to Erd{\H o}s' question. As a reminder, for $N\leq 4$, all faces of $\Pi_{N}$ are centered, whereas for $N\geq 5$, not all faces are centered (see \cref{th:BirkhoffMon}).

\begin{theorem}\label{th:erd}
   For $A\in\R^{N\times N}$, the following are equivalent:
   \begin{enumerate}
   \item $A$ is an Erd{\H o}s matrix,
   \item $A=\proj_{\aff F} (0)$ for some centered face $F$ of $\Pi_{N}$, and $A\neq \proj_{F'} (0)$ for any non-centered face  $F'$ containing~$F$.
   \end{enumerate} 
\end{theorem} 

\begin{proof} $(i) \Rightarrow (ii)$: Let $A$ be an Erd{\H o}s matrix and $\{P_{1}, \dots, P_{k}\}= \argmax_{P\in \ext \Pi_N}  \langle P,A \rangle$.  Then we have $\langle A,A \rangle = \langle P_{i} ,A \rangle$ for $i=1,\dots,k$ and $\langle A,A \rangle > \langle P ,A \rangle$ for all $P\in \ext \Pi_N \setminus \{P_{1}, \dots, P_{k}\}$. Hence, $$F:=\coh (P_{1}, \dots, P_{k})=\argmax_{P\in \Pi_N}  \langle P,A \rangle $$ 
  is a face containing $A$. As $\langle A,A \rangle = \langle P_{i} ,A \rangle$ for $i=1,\dots,k$, we deduce  $A=\proj_{\aff F}(0)$; cf.\ \cref{Lemma:projset}. Recalling that $A$ is doubly stochastic, this implies $A\in (\aff F) \cap \Pi_N =  F$, and hence that $F$ is centered. 
  
  To see the second part of (ii), suppose that $A= \proj_{F'} (0)$ for some face $F'$. When $F'$ is not centered, it follows that $A\neq \proj_{\aff F'} (0)$. By \cref{Lemma:projset}, the combination of both facts implies that
  $$0< \langle A, P-A \rangle =\langle A,P\rangle -\|A\|^2 $$
  for some $P\in \ext F'$, contradicting that  $A$ is Erd{\H o}s. (We did not use the condition that $F'\supset F$. Thus, the theorem remains true if in (ii) we require $A\neq \proj_{F'} (0)$ for all non-centered faces $F'$ of $\Pi_N$, and not just the ones containing $F$.)

  $(ii) \Rightarrow (i)$: Assume that $A=\proj_{\aff F} (0)$ for some centered face $F$ of $\Pi_{N}$, hence also $A\in F$, and assume that $A$ is not Erd{\H o}s. We shall construct a face $F'\supset F$ with $A= \proj_{F'} (0)$ that is not centered. Indeed, $A=\proj_{\aff F} (0)$ implies in particular that
  \begin{equation}
      \label{eq:Fsuper}F \subset \{\pi\in \Pi_N: \langle A, \pi-A \rangle = 0\}.
  \end{equation}
  We claim that  there exists a face $F'$ of $\Pi_N$ such that $F\subset F'$, 
  \begin{equation}
      \label{eq:CenteredFprima}
   F' \subset  \{\pi\in \Pi_N: \langle A, \pi-A \rangle\geq  0\}   \quad {\rm and}\quad \{\pi\in \Pi_N: \langle A, \pi-A \rangle> 0\} \cap F'\neq \emptyset.
  \end{equation}
  This claim will complete the proof: Indeed, as $A\in F\subset F'$, the first part of~\eqref{eq:CenteredFprima} implies that $A= \proj_{F'} (0)$. If $F'$ were centered, this would yield $A= \proj_{\aff F'} (0)$, contradicting the second part of~\eqref{eq:CenteredFprima}.

  It remains to prove the claim. Let 
  $\{\pi:\langle B, \pi -A\rangle =0\} $ 
  be a supporting hyperplane such that 
   \begin{equation}
     \label{eq:Bhyper}  
  \Pi_N\subset  \{\pi:\langle B, \pi -A\rangle \leq 0\} \quad {\rm and}\quad  \Pi_N\cap  \{\pi:\langle B, \pi -A\rangle = 0\} =F.
  \end{equation}
  We define
  $$H_\alpha = \{\pi:\langle \alpha B  + (1-\alpha) A, \pi -A\rangle \leq  0\}, \quad   \alpha\in[0,1]$$
 and the function 
 $$ [0,1]\ni \alpha \mapsto \Phi(\alpha):=\begin{cases}
     1& {\rm if } \ \Pi_N\subset H_\alpha,  \\
     0& {\rm otherwise.}
 \end{cases}$$
Let $\pi\in\Pi_N$. As $\langle B, \pi -A\rangle \leq 0$, there are two cases: if $\langle A, \pi -A\rangle \leq 0$, then $\pi\in H_\alpha$ for all $\alpha\in[0,1]$, whereas if $\langle A, \pi -A\rangle > 0$, then $\alpha\mapsto \langle \alpha B  + (1-\alpha) A, \pi -A\rangle$ is nonincreasing. It follows that $\alpha\mapsto H_\alpha \cap \Pi_N$ is nondecreasing (for inclusion) and hence that $\Phi$ is nondecreasing. Notice also that $\Phi$ is right-continuous (by continuity of the inner product) with $\Phi(1)=1$, and that
 \begin{equation}
     \label{eq:relintHiperplane}
      \{\pi\in \Pi_N: \langle A, \pi-A \rangle \leq 0\} \subset H_\alpha \quad \text{for all }\alpha\in [0,1].
 \end{equation} 
In view of \eqref{eq:Fsuper} and \eqref{eq:Bhyper}, the latter implies that 
 \begin{equation}
     \label{eq:Hboundary}
     F \,\subset\, (\rbd H_{\alpha} )\cap \Pi_N\,\subset \, \{\pi\in \Pi_N: \langle A, \pi-A \rangle\geq  0\}\quad \text{for all }\alpha\in [0,1].
 \end{equation}
Define 
$$
\alpha^*:=\min\{\alpha\in [0,1]: \Phi(
\alpha)=1\}\quad\mbox{and} \quad
F':= (\rbd H_{\alpha^*} )\cap \Pi_N,
$$
where the minimum is attained thanks to right-continuity and $\Phi(1)=1$. 
The fact that $\Phi(\alpha^*)=1$ yields that $F'$ is a face of $\Pi_N$. Moreover, \eqref{eq:Hboundary} for $\alpha=\alpha^*$ shows $F\subset F'$ and the first part of \eqref{eq:CenteredFprima}. 
It remains to show the second part of \eqref{eq:CenteredFprima}. Suppose for contradiction that it fails, then by~\eqref{eq:relintHiperplane} we have $F' \subset H_\alpha$ for all $\alpha\in[0,1]$. Thus
$$ 
F':=(\rbd H_{\alpha^*} )\cap \Pi_N \subset H_\alpha \cap \Pi_N \quad\mbox{for all } \alpha\in[0,1].
$$
On the other hand, the definition of $\alpha^*$ implies%
\footnote{Specifically, let $S'\subset\ext\Pi_N$ be the set of extreme points of $\Pi_N$ that are contained in $F'$, or equivalently are contained in $\rbd H_{\alpha^*}$. Note that all of $\ext\Pi_N$ is contained in $H_{\alpha^*}$ as $\Phi(\alpha^*)=1$. As $\ext\Pi_N$ is a finite set and $\alpha\mapsto H_\alpha$ is continuous, there exists $\epsilon>0$ such that $(\ext\Pi_N\setminus S') \subset H_\alpha \cap \Pi_N$  for all $\alpha\in (\alpha^*-\epsilon,\alpha^*)\cap[0,1]$. We choose such an $\epsilon$. If $\alpha\in (\alpha^*-\epsilon,\alpha^*)\cap[0,1]$, then as $\Phi(\alpha)=0$, we know that $\ext\Pi_N\not\subset H_\alpha$, and it follows that $S'\not\subset H_\alpha$; i.e., $F'\not\subset H_\alpha$.}
$$(\rbd H_{\alpha^*} )\cap \Pi_N \not\subset H_\alpha \cap \Pi_N \quad\mbox{for all $\alpha\in (\alpha^*-\epsilon,\alpha^*)\cap[0,1]$}
$$
for some $\epsilon>0$. 
The two displays form a contradiction unless $\alpha^*=0$. Hence $\Phi(0)=1$; that is, $\Pi_N\subset H_0$. The latter means that $\langle A, \pi -A\rangle \leq 0$ for all $\pi\in\Pi_N$, which in turn implies that~$A$ is Erd{\H o}s. This contradiction completes the proof.
\end{proof}

As a consequence of
\cref{le:BirkhoffMonotone}, we deduce the following simpler statement for $N\leq 4$.
\begin{corollary}
     For $A\in\R^{N\times N}$ and $N\leq 4$, the following are equivalent:
   \begin{enumerate}
   \item $A$ is an Erd{\H o}s matrix,
   \item $A=\proj_{\aff F} (0)$ for some face $F$ of $\Pi_{N}$. 
   \end{enumerate} 
\end{corollary}

Arguably \cref{th:erd} is a only a partial answer to Erd{\H o}s' question, as it is not straightforward to list all centered faces of $\Pi_{N}$ when $N$ is large. To illustrate how \cref{th:erd} can be useful, we use this geometric point of view to answer a question recently posed in \cite[Question~7.2]{BouthatMashreghiMorneauGuerin.24}: Do all Erd{\H o}s matrices have rational entries? (The main contribution of \cite{BouthatMashreghiMorneauGuerin.24} is to list explicitly all $3\times3$ Erd{\H o}s matrices, and it is observed that they have rational entries. Of course this list can also be recovered using our description.) \Cref{th:erd} allows us to answer this affirmatively in a general and elegant way.

\begin{corollary}\label{co:erd}
  Erd{\H o}s matrices have rational entries.
\end{corollary} 

\begin{proof}
  Let $A$ be an Erd{\H o}s matrix. By \cref{th:erd}, $A=\proj_{\aff F} (0)$ for some face $F$ of $\Pi_{N}$. Let $P_{1}, \dots, P_{k}$ be the vertices of $F$ and note that 
  $$
    A=\proj_{\aff F} (0)=\proj_{\aff \{0,P_{2}-P_{1},\dots, P_{k}-P_{1}\}} (-P_{1}) = \proj_{\lin  \{P_{2}-P_{1},\dots, P_{k}-P_{1}\}} (-P_{1}),
  $$
  the projection of $-P_{1}$ onto the linear span of $P_{2}-P_{1},\dots, P_{k}-P_{1}$. Eliminating some of the latter matrices if necessary, we may assume that they are linearly independent. If we see $N\times N$ matrices as vectors of length $N^{2}$ by stacking their columns, the projection admits the following formula, well known as the solution of the least squares linear regression problem:
  \begin{equation}\label{eq:regression}
     A=(X^{\top} X)^{-1} X^{\top}(-P_{1}),
  \end{equation}
  where $X$ is the $N^{2}\times (k-1)$ matrix with columns $P_{2}-P_{1},\dots, P_{k}-P_{1}$. Because $P_{i}$ are permutation matrices, $X$ has integer entries and it follows that $(X^{\top} X)^{-1}$ has rational entries (as can be seen from Cramer's rule for computing the inverse). It is now apparent from~\eqref{eq:regression} that $A$ also has rational entries.
\end{proof}

\newcommand{\dummy}[1]{}

\end{document}